\documentclass[twoside, a4paper, nofootinbib,11pt]{article}
\usepackage{amssymb}
\usepackage{IEEEtrantools}
\usepackage[mathscr]{eucal}
\usepackage[dvips]{graphicx}
\usepackage{amsmath}
\usepackage{amsthm}
\usepackage{pstricks}
\usepackage{caption}
\def \ni{\noindent}

\newcommand{\be}{\begin{equation}}
\newcommand{\ee}{\end{equation}}
\newcommand{\ben}{\begin{equation*}}
\newcommand{\een}{\end{equation*}}
\newcommand{\bes}{\begin{eqnarray}}
\newcommand{\ees}{\end{eqnarray}}
\newcommand{\besn}{\begin{IEEEeqnarray*}{rCl}}
\newcommand{\eesn}{\end{IEEEeqnarray*}}

\newcommand{\txt}{\textrm}

\newtheorem*{theorem*}{Theorem}

\newtheorem*{definition*}{Definition}
\newtheorem*{lemma*}{Lemma}

\title{Bump Functions With Monotone Fourier Transforms Satisfying Decay Bounds}
\author{T. Tlas}
\date{}

\begin{document}
\maketitle
\thispagestyle{empty}

\begin{abstract}
\ni The existence of a smooth, nonnegative, compactly supported function with monotone (on the half-line) Fourier transform satisfying two-sided decay bounds is demonstrated.
\end{abstract}

\vspace{0.5cm}

The uncertainty principle in harmonic analysis roughly states that a function and its Fourier transform cannot both be made to decay too fast. In particular, it is easy to see that it is impossible for a compactly supported function to have an exponentially decaying Fourier transform, as this would force the function to be analytic. The Beurling-Malliavin multiplier theorem (see e.g. \cite{7th}) guarantees the existence of  smooth, compactly supported functions whose Fourier transforms have sub-exponential decay (i.e. an exponential of a power less than one). However, one may need further information, e.g. a lower bound of the transform or whether it is monotonic. \newline

The goal of this brief note is to prove the following

\begin{theorem*}
For any $\delta \in (0,1)$ and any $C >0$  there is a function $F(x)$ which is $C^\infty$, real, even, nonnegative, supported in $[-1,1]$ and whose Fourier transform $\widehat{F}(k)$ is monotone decreasing for $k \geq 0$ and satisfies the following double inequality 
\ben
  e^{- (1+\epsilon) C k^\delta  } \lesssim  \widehat{F}(k)  \lesssim  \, \, e^{- (1- \epsilon) C k^\delta},
\een

for any $\epsilon > 0$.

\end{theorem*}

Above we've used the notation $f_1(k) \lesssim f_2(k)$ to mean $f_1(k) \leq c \, f_2(k)$ for some constant $c$. Our convention for the Fourier transform is $\widehat{f}(k) = \int e^{-2 \pi i kx} f(x) dx$. To reduce clutter in some formulas below, we will set $\tilde{k} = 2 \pi k$.   \newline

We first prove a lemma of interest\footnote{See \cite{bump} for a related heuristic discussion.} in its own right. 

\begin{lemma*}
For $A,B \in (0, \infty)$, let the smooth function $\phi_{A,B}: \mathbb{R} \to \mathbb{R}$ be given by

\ben
\phi_{A,B}(x) =
\begin{cases}
e^{-\frac{B}{(1- x)^A}} e^{-\frac{B}{(1+x)^A}} & |x| < 1 \\
0 & |x| \geq 1
\end{cases}
\een

then, as $k \to \infty$, its Fourier transform $\widehat{\phi_{A,B}} (k)$ is asymptotic to a constant multiple of

\ben
\frac{  \cos( \tilde{k} - \alpha \tilde{k}^{\frac{A}{A+1}})   e^{  - \beta \tilde{k} ^{  \frac{A}{ A+1  }  }}      }      {    \tilde{k}^{\frac{A+2}{2A+2} }    }
\een

for constants $\alpha, \beta$. Moreover, for a fixed $A$, $B$ can be chosen to give any prescribed positive value of $\beta$.

\end{lemma*}

\begin{proof}

\begin{figure}
\begin{pspicture}(13,6)
\psline{->}(0,3)(6,3)
\psline{->}(3,0)(3,6)
\pscircle(0.5,3){0.1}
\pscircle(5.5,3){0.1}
\qdisk(3, 2){3pt}
\psline[linewidth=2pt]{->}(0.6,2.9)(3,2)(5.4,2.9)
\rput(0.5, 3.5){-1}
\rput(5.5, 3.5){1}
\rput(1,1 ){$-i \tan\Big (\frac{\pi}{2(A+1)} \Big )$}
\rput( 4, 2.1 ) {$\gamma$}
\psline[linestyle=dotted, linearc=2]{->}(0.9,1.4)(2,2)(2.8, 2)

\psline{->}(7,1)(12,1)
\psline{->}(7.5, 0)(7.5, 6)
\pscircle(7.5,1){0.1}
\qdisk( 11 , 2.4 ) {3pt}
\psline[linewidth=2pt]{->}(7.6,1.05)(10.9, 2.35)
\rput(9,2){$\gamma_{\tilde{k}}$}
\rput( 10,5 ) {$\tilde{k}^\frac{1}{A+1} \bigg (1 + i \tan \Big (\frac{\pi}{2 (A+1)}\Big ) \bigg ) $   }
\psline[linestyle=dotted, linearc=2]{->}(9,4.5)(10,4)(10.5, 3)( 10.9 , 2.5)
\end{pspicture}

\caption*{The contours $\gamma$ and $\gamma_{\tilde{k}}$.}
\end{figure}
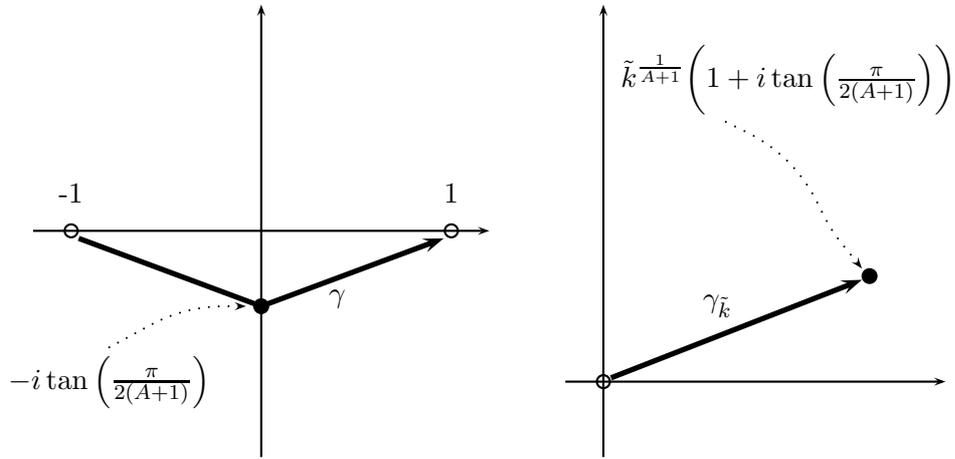

We need to estimate 

\ben
\int_{-1}^1 e^{- i \tilde{k} x} \phi_{A,B}(x) dx .
\een

Choosing the principal branch of the logarithm (to define the exponents in $\phi_{A,B}$), we see that the integrand is an analytic function in the region $|\Re(z) | < 1, \Im(z) \leq 0$. We can thus replace integration along the real axis with the integral along the contour $\gamma$ shown in the figure above. The reason for this choice will be apparent momentarily. Note that no trouble arises at the endpoints of the contour as the integrand remains bounded (in fact it tends to zero) as we approach them. \newline

It is easy to see that the integral along the left segment is the conjugate of the one on the right. Doing now the change of variables $z \to \tilde{k}^\frac{1}{A+1} (1-z) $ we see that the integral above is equal to

\ben
2 \Re \bigg (  -\frac{e^{-i\tilde{k}}}{\tilde{k}^\frac{1}{A+1}      }     \int_{    \gamma_{\tilde{k} }  } e^{ \tilde{k}^\frac{A}{A+1} g(z)    } h_{\tilde{k}}(z) dz    \bigg )
\een

where 

\besn
g(z) & = & iz - \frac{B}{z^A} \\
h_{\tilde{k}}(z) &  =   &   e^{ - B \big (  2 - \tilde{k}^\frac{-A}{A+1}  z \big )^{-A}     }
\eesn

and $\gamma_{\tilde{k}}$ is the contour shown in the figure. \newline

It is straightforward to verify that, provided $\tilde{k}$ is large enough, then $g$ has a single critical point on the contour located at $z_0 = (AB)^\frac{1}{A+1} e^\frac{ i \pi}{2(A+1)}$. Moreover, $\Re(g)$ has a maximum there. We can now deform the contour in a neighbourhood around $z_0$ so that it follows the path of steepest descent there and apply the steepest descent method\footnote{See e.g. \cite{marsden}. Note that one needs to extend the result there as $h_{\tilde{k}}$ and $\gamma_{\tilde{k}}$ both depend on $\tilde{k}$, but such an extension is trivial in this case. } to obtain that, up to a constant

\ben
 \int_{    \gamma_{\tilde{k} }  }   e^{ \tilde{k}^\frac{A}{A+1} f(z)    } g_{\tilde{k}}(z) dz \sim \frac{  \exp \bigg [ \tilde{k}^\frac{A}{A+1} \big ( i  e^\frac{i \pi}{2(A+1)} (AB)^\frac{1}{A+1} (1 + \frac{1}{A})                                \big )     \bigg ] }{\tilde{k}^\frac{A}{2(A+1)}}
\een

Putting everything together, we have what we want with

\besn
\alpha  & =  &  (AB)^\frac{1}{A+1} \Big ( 1+ \frac{1}{A} \Big ) \cos   \Big(\frac{\pi}{2A+2}\Big)      \\
\beta & = &  (AB)^\frac{1}{A+1} \Big (1 + \frac{1}{A}\Big ) \sin \Big(\frac{\pi}{2A+2}\Big).
\eesn

Clearly, for a fixed $A$, $\beta$ takes all values in $(0, \infty)$ by varying $B$.
\end{proof}

We are now ready to prove the theorem.

\begin{proof}[Proof of theorem]
Let $\phi(x)$ be a real, even, nonnegative, $C_0^\infty$ function whose support is contained in $[-\frac{1}{2}, \frac{1}{2}]$. Thus, $\widehat{\phi}(k)$ is real and even. It follows that $\phi \ast \phi(x)$ is a real, even, nonnegative, $C_0^\infty$ function whose support is contained in $[-1, 1]$, and moreover, $\widehat{\phi \ast \phi} = \big (\widehat{\phi}(k)\big)^2 \geq 0$. Let $\psi(x) = i (\phi \ast \phi)' (x)$. We then have that $\widehat{\psi}(k) = - (2 \pi) k \big (\widehat{\phi}(k)\big)^2$. Therefore, $\hat{\psi}(k)$ is a real, odd function, which is nonpositive for $k \geq 0$. Moreover, by Paley-Wiener's theorem (see e.g. \cite{reed}) we have that $\widehat{\psi}(k)$ has an extension to an entire function such that, for every natural $N$, there is a constant $C_N$ such that 

\ben
|\widehat{\psi}(k)| \leq C_N \frac{e^{ 2 \pi | \Im(k) |}}{(1 + |k|)^N}.
\een

Now, set 

\ben
I = - \int_0^\infty \widehat{\psi}(t) dt,
\een

and define the function\footnote{At the moment the hat is purely notational. It will be justified below.}
 $\widehat{f}(k)$ by
\ben
\widehat{f}(k) = I + \int_\gamma \widehat{\psi}(z) dz,
\een

where $\gamma$ is some curve in the complex plane going from 0 to $k$. We see at once that $\widehat{f}(k)$ is an entire function whose derivative is $\widehat{\psi}(k)$ and that, if $k$ is real, then $\widehat{f}(k)$ is a real, even and nonnegative Schwartz function. \newline

Now choose $\gamma$ to be a straight line segment $\gamma_1$ along the real axis from 0 to $\Re (k)$, followed by a vertical one $\gamma_2$ to $k$. We then have

\ben
\Big | \int_{\gamma_1} \widehat{\psi}(z) dz  \Big | \leq  \int_0^{|\Re(k)|} | \widehat{\psi}(t)| dt
 \leq  C_2 \int_{-\infty}^\infty \frac{1}{(1 + |t|)^2} dt 
 <  \pi C_2.
\een

Also,

\besn
\Big | \int_{\gamma_2} \hat{\psi}(z) dz  \Big | & \leq & \int_0^{|\Im (k)|}   | \widehat{\psi}(\Re(k) + i t)| dt \\
& \leq & C_2 \int_0^{|\Im (k)|} \frac{e^{2 \pi t}}{(1 + | \Re(k) + it|)^2} \leq C_2 e^{2 \pi |\Im(k)|} |k| .
\eesn

Putting it all together, we have that $\widehat{f}(k)$ is an entire function which satisfies

\ben
| \widehat{f}(k) | \leq  \pi C_2 ( 1  + |k|) e^{2 \pi |\Im (k)|}.
\een

It then follows, again by Paley-Wiener, that there is a Schwartz distribution $f(x)$, supported in $[-1, 1]$ whose Fourier transform has $\widehat{f}(k)$ as its entire extension (hence the notation). However, we showed above that $\widehat{f}(k)$ is a Schwartz function when restricted to the real line, which implies that $f(x)$ is in fact a Schwartz function. Thus, $f(x)$ is a real, even, $C_0^\infty$ function supported in $[-1, 1]$ whose Fourier transform's derivative is $\widehat{\psi}(k)$, and is thus negative for positive $k$'s. Let us now prove the decay estimates.\newline

Suppose $C >0$ and $\delta \in (0,1)$ are given. We shall take $\phi(x) = \phi_{A,B}(2x)$, choosing $A$ and $B$ such that up to an overall constant

\ben
\widehat{\phi_{A,B}}\bigg(\frac{k}{2}\bigg) \sim \frac{  \cos(\frac{\tilde{k}}{2} - \frac{\alpha}{2^\delta} \tilde{k}^\delta) e^{- \frac{C k^\delta}{2}}   }{    k^\frac{A+2}{2A + 2}      } .
\een

It follows at once that for $k \geq 0$

\besn
\widehat{f}(k) \lesssim \int_k^\infty  e^{- C t^\delta}     dt \lesssim e^{ (1-\epsilon) C k^\delta}
\eesn

for any $\epsilon >0$, and we have the upper bound we need.\newline

To obtain the lower bound, let $Z$ be the set of zeroes of the function $\cos^2\big(\frac{\tilde{k}}{2} - \frac{\alpha}{2^\delta} \tilde{k}^\delta \big)$, and $\tilde{Z}$ the preimage by this function of $[0, \frac{1}{4}]$. Given $k$, let $\{k_n \}_{n=1}^\infty$ be the sequence of elements of $Z \cap (k, \infty)$ arranged in increasing order, and let $k_0$ be the largest element in $Z$ less or equal to $k$. Since $\delta < 1$, it is clear, by considering the derivative of the argument of the cosine, that $k_{n+1} - k_n \to 2 \pi$ as $n \to \infty$. It also follows that $\tilde{Z}$ consists of a collection of disjoint intervals, each containing a single $k_n$, with their lengths tending to $\frac{2 \pi}{3}$.     \newline

We now have that, for all sufficiently large $k$, and for any $\epsilon > 0$

\besn
\widehat{f}(k) & \gtrsim & \int_k^\infty \frac{  \cos^2(\frac{t}{2} -\frac{ \alpha}{2^\delta} t^\delta) e^{- C t^\delta}    }{   t^\frac{1}{A+1}     }   dt \\
& \gtrsim & \int_{\tilde{Z}^c \cap [k, \infty)}  \frac{e^{- C t^\delta}}{t^\frac{1}{A+1}} dt \\
& \gtrsim & \int_{\tilde{Z}^c \cap [k, \infty)} e^{- (1+ \frac{\epsilon}{2}) C t^\delta} dt .
\eesn

Define $\lambda \in [0,1]$ via $k = (1- \lambda) k_0 + \lambda k_1$ and let $I_n = [ (1-\lambda) k_{n-1} + \lambda k_n, (1-\lambda) k_n + \lambda k_{n+1}]$. We then see that

\besn
\int_{\tilde{Z}^c \cap [k, \infty)} e^{- (1+ \frac{\epsilon}{2}) C t^\delta} dt & = & \sum_{n=1}^\infty \int_{ \tilde{Z}^c \cap I_n   }  e^{- (1+ \frac{\epsilon}{2}) C t^\delta} dt  \\
& \geq &  \sum_{n=1}^\infty \int_{\big [ \frac{k_{n-1} + k_n   }{2}, (1-\lambda) k_n + \lambda k_{n+1} \big  ] } e^{- (1+ \frac{\epsilon}{2}) C t^\delta} dt\\
& \geq  & \sum_{n=1}^\infty \int_{I_n} e^{-(1 + \epsilon)C t^\delta} dt \\
& = & \int_k^\infty  e^{-(1 + \epsilon)C' t^\delta} dt \gtrsim e^{-(1 + \epsilon)C t^\delta},
\eesn 

where to go to the penultimate line we've used the fact that 

\ben
\frac{     \int_x^{x+L(x)} e^{ - (1+  \epsilon) C t^\delta    }    dt  }{         \int_{x + \frac{L(x)}{2} }^{x + L(x)} e^{- (1 + \frac{\epsilon}{2}    )C t^\delta  }   dt        } \to 0 \quad \txt{as} \quad x \to \infty
\een

as long as $L(x)$ is bounded away from 0 and $\infty$, which is the case above since $L(x)$ will be contained in a narrow range around $2 \pi$. \newline

We have thus obtained a function $f$ satisfying all the requirements of the theorem with the possible exception of nonnegativity. \newline

Let $F(x) = f^2(x)$, then $F$ is $C^\infty$, real, even, nonnegative and is supported in $[-1,1]$. What remains is showing the monotonicity of its Fourier transform and the decay bounds.\newline

Let $k \geq 0$. We then have that

\ben
\frac{d \widehat{F}(k)}{dk} 
  =  \int_0^\infty \Big ( \widehat{f}(\kappa - k) -  \widehat{f}(\kappa + k) \Big )  \frac{d\widehat{f}}{d \kappa}(\kappa) d \kappa .
\een

Since $\widehat{f}$ is monotone and even, the term in brackets is positive  ($|\kappa - k| \leq \kappa + k$) while $ \frac{d\widehat{f}}{d \kappa}(\kappa)$ is negative. Therefore $\widehat{F}$ is monotone decreasing on the positive axis. \newline

Now, note that

\besn
\widehat{F}(k) & = & \int \widehat{f}(k - \kappa) \widehat{f}(\kappa) d\kappa  \\
& \lesssim & \int e^{ - (1- \epsilon) C | k - \kappa|^\delta   } e^{ - (1- \epsilon) C |\kappa|^\delta  } d \kappa  \\
& = & k \int e^{   - (1-\epsilon) C k^\delta \big ( | 1- \chi|^\delta + | \chi|^\delta \big ) } d \chi .
\eesn

It is easy to show, essentially by an application of Laplace's method\footnote{A little care is needed since the maxima of the exponent which dominate the asymptotics are located at 0 and 1 where the exponent is not regular. However, the extension of the method needed is completely straightforward in this case.}, that the last expression is asymptotic to

\ben
\frac{   2 e^{- (1- \epsilon) C k^\delta} \bigg ( \int_{-\infty}^\infty e^{-|\chi|^\delta} d \chi  \bigg ) }{  \big ( (1- \epsilon) C \big )^{\frac{1}{\delta}}  },
\een

which in turn gives that

\ben
\widehat{F}(k) \lesssim e^{(1 - \epsilon) C k^\delta},
\een

and we have the decay estimate from the right. The one from the left is demonstrated in a similar way, which concludes the proof.
 \end{proof}

\textbf{Acknowledgments:} The author would like to thank J. Merhej and B. Shayya for reading a preliminary version of this paper and for numerous comments which greatly improved its readability. Additionally, the author would like to thank J. Steif for posing the question whether the function constructed can be taken to be positive.

\texttt{{\footnotesize Department of Mathematics, American University of Beirut, Beirut, Lebanon.}
}\\ \texttt{\footnotesize{Email address}} : \textbf{\footnotesize{tamer.tlas@aub.edu.lb}}

\end{document}